\documentclass[12pt,reqno]{amsart}
\usepackage[english]{babel}
\usepackage{amsmath,amsfonts,amssymb,amsthm,amscd,latexsym}
\usepackage[T1]{fontenc}
\usepackage[utf8]{inputenc}

\usepackage{tikz}
\usetikzlibrary{arrows}
\usepackage{color}
\usepackage{cite}
\usepackage{multirow}
\usepackage{cite}

\usepackage{ifpdf}
\ifpdf
 \usepackage[colorlinks,final,backref=page,hyperindex]{hyperref}
\else
  \usepackage[colorlinks,final,backref=page,hyperindex,hypertex]{hyperref}
\fi

\newtheorem{theorem}{Theorem}[section]

\newtheorem{corollary}[theorem]{Corollary}
\theoremstyle{definition}

\theoremstyle{remark}
\newtheorem{remark}[theorem]{Remark}
\numberwithin{equation}{section}

\begin{document}

\setcounter{page}{1}

\title[Spectra and eigenspaces of non-normal Cayley graphs]{Spectra and eigenspaces of non-normal Cayley graphs}

\author{Yang Chen}
\address{School of Science,
Jimei University, Xiamen 361021, PR China}
\email{chenyang1729@hotmail.com}

\author{Xuanrui Hu}
\address{School of Science,
Jimei University, Xiamen 361021, PR China}
\email{1950113537@qq.com}

\date{}
\maketitle

\begin{abstract} In this paper, we construct some non-normal Cayley graphs and explicitly provide their spectra and eigenspaces using representation theory of finite groups.

\


\end{abstract}

\section{Introduction}

The spectra and eigenspaces of the adjacency matrices of graphs are important algebraic invariants. Graphs with strong symmetry properties are of particular interest in this regard. Lov\'{a}sz \cite{Lov} determined the spectrum of a graph with transitive automorphism group in terms of group characters. Babai \cite{Bab} gave a more handy formula of the spectrum of a Cayley graph by different methods. However, their formulas cannot enable us to explicitly compute eigenvalues of all Cayley graphs, and does not gave eigenspaces. It is worth mentioning that the adjacency matrix of a Cayley color graph is the transpose of a special group matrix, and its characteristic polynomial is the group determinant proposed by Frobenius \cite{Fro1, Fro2}, which motivated Frobenius to introduce characters of finite groups (see \cite{Cur}). From this, we can analyze the spectrum of a Cayley color graph by Frobenius determinant theorem, which describes the factorization of the group determinant for a finite group. In fact, Lov\'{a}sz and Babai's results are contained in the Peter-Weyl theorem (see Theorem \ref{thm21}) generalizing the significant facts about the decomposition of the regular representation of any finite group, discovered by Frobenius and Schur. A Cayley graph is called \textit{normal} if the connection set is closed under conjugation, that is, it is the union of conjugacy classes of the group. Note that the term normal Cayley graph has been used also by Xu \cite{Xu} to denote a different graph construction. The spectra of the normal Cayley graphs can be explicitly calculated using the irreducible characters of the underlying group, see Diaconis and Shahshahani \cite{DiaSha} and Zieschang \cite{Zie}. These results were then applied to other research areas. In 2022, Sin and Sorci \cite{SinSor} studied continuous-time quantum walks on normal Cayley graphs of extraspecial groups. In 2023, \'{A}rnad\'{o}ttir and Godsil \cite{ArnGod} proved an upper bound on the number of pairwise strongly cospectral vertices in a normal Cayley graph in terms of the multiplicities of its eigenvalues. However, the discussion becomes more complicated in the non-normal case, and there are few relevant research results. Since the dimensions of the irreducible representations of the dihedral group are less than or equal to $2$, Cao, Chen and Ling \cite{CCL} could study perfect state transfer on non-normal Cayley graphs of dihedral groups by direct calculations. We would like to give a formular for explicitly calculating spectra and eigenspaces of all non-normal Cayley graphs, but this seems impossible. From a basic algebraic point of view, we consider using group extension to study non-normal Cayley graphs. That is, normal Cayley graphs of subgroups and quotient groups are used to construct non-normal Cayley graphs. In this paper, we construct some non-normal Cayley graphs for split group extension and exactly give spectra and eigenspaces of these graphs.

\section{Preliminaries}

Firstly, we recall some definitions and related results concerning Cayley color graphs and representation theory of finite groups. Let $\mathbb{C}$ denote the field of complex numbers. A \textit{color graph} is a pair $\Gamma= (V; c)$, where $V= \{v_1, \ldots, v_n\}$ is the vertex set and $c: V\times V\rightarrow \mathbb{C}$ is a function, associating a color with each directed edge. Non-edges can be interpreted as ones having color 0. The \textit{adjacency matrix} of $\Gamma$ is defined as
$$\text{adj}(\Gamma)= (a_{ij})_{i, j=1}^n, \text{ where } a_{ij}= c(v_i, v_j).$$
The \textit{spectrum} of $\Gamma$ is that of its adjacency matrix. For $G$ a group and $\alpha: G\rightarrow \mathbb{C}$ a function, the \textit{Cayley color graph} $\Gamma(G; \alpha)$ is defined on the vertex set $V(\Gamma)= G$ by $c(g, g')= \alpha(g'g^{-1})$, $g, g'\in G$. If $\alpha(g)\in \{0, 1\}, g\in G$ and the set $S= \{g| \alpha(g)= 1\}$ generates $G$, then $\Gamma$ is a \textit{Cayley digraph} of $G$ with connection set $S$, denoted by $\Gamma(G, S)$. When $S$ is inverse-closed (i.e. $S= S^{-1}$) and does not contain the identity element, $\Gamma(G, S)$ can be represented as a simple undirected graph called the \textit{Cayley graph}.

The reader is referred to \cite{Ser} for representation theory of finite groups. Let $G$ be a finite group of order $n$. We denonte by $\mathbb{C}G$ the algebra of $G$ over $\mathbb{C}$. Each element $f$ of $\mathbb{C}G$ can then be uniquely written in the form
$$f= \sum_{g\in G} f_g g, f_g\in \mathbb{C},$$
and multiplication in $\mathbb{C}G$ extends that in $G$. The group algebra $\mathbb{C}G$ is isomorphic to its dual $L(G)$, which is the set of complex-valued functions on $G$ with the convolution product. For $f\in L(G)$ and $\rho$ a representation of $G$, the \textit{Fourier transform} of $f$ with respect to $\rho$ is defined to be
$$\mathcal{F}_\rho(f)= \hat{f}(\rho)= \sum_{g\in G} f(g)\rho(g).$$
Let $\rho_{\rm reg}$ denote the left regular representation of $G$. Clearly, the adjacency matrix of the Cayley color graph $\Gamma(G; \alpha)$ is
$$\text{adj}(\Gamma(G; \alpha))= (\hat{\alpha}(\rho_{\rm reg}))^T= \sum_{g\in G} \alpha(g)(\rho_{\rm reg}(g))^T,$$
where $T$ indicates the usual transpose. The following result can be obtained by the Peter-Weyl theorem for a finite group, see also \cite[Chapter 3, E]{Dia}.

\begin{theorem}\label{thm21} Suppose that $G= \{g_1, \ldots, g_n\}$ is a finite group of order $n$ with irreducible unitary representations $\rho_1, \ldots, \rho_r$ of degrees $d_1, \ldots, d_r$. Let
$$(\rho_k)_{ij}= \sqrt{\frac{d_k}{n}}((\rho_k(g_1))_{ij}, \ldots, (\rho_k(g_n))_{ij})^T, 1\leq i, j\leq d_k, 1\leq k\leq r,$$
and
$$P= ((\rho_1)_{11}, \ldots, (\rho_1)_{d_1 1}, (\rho_1)_{12}, \ldots, (\rho_1)_{d_1d_1}, (\rho_2)_{11}, \ldots, (\rho_2)_{d_2d_2}, \ldots, (\rho_r)_{d_rd_r}).$$
If $f$ is a complex-valued function on $G$, then the Fourier transform of $f$ at $\rho_{\rm reg}$ can be written as
$$\hat{f}(\rho_{\rm reg})= \bar{P} {\rm diag}(I_{d_1}\otimes \hat{f}(\rho_1), \ldots, I_{d_r}\otimes \hat{f}(\rho_r))\bar{P}^{-1},$$
where $\bar{}$ denotes the conjugate.
\end{theorem}

A function $f$ on $G$ is called a \textit{class function} if $f(gg'g^{-1})= f(g')$ for all $g, g'\in G$. If $\rho$ is irreducible of character $\chi$, by Schur's lemma we know that $\hat{f}(\rho)$ is a homothety of ratio $\lambda$ given by
$$\lambda= \frac{1}{\chi(1)}\sum_{g\in G}f(g)\chi(g).$$
For a class function $\alpha$ and irreducible characters $\chi_1, \ldots, \chi_r$ of $G$, the spectrum of the Cayley color graph $\Gamma(G; \alpha)$ is given by
$$\lambda_k= \frac{1}{\chi_k(1)}\sum_{g\in G}\alpha(g)\chi_k(g), 1\leq k\leq r,$$
and the corresponding eigenspaces are
$$(\rho_k)_{ij}= \sqrt{\frac{\chi_k(1)}{n}}((\rho_k(g_1))_{ij}, \ldots, (\rho_k(g_n))_{ij})^T, 1\leq i, j\leq \chi_k(1), 1\leq k\leq r$$
independent of the choice of $\alpha$.

Let $S$ be the union of some conjugacy classes of $G$ and the class function
$$\alpha(g)= \left\{\begin{matrix}
1, & g\in S, \\
0, & g\notin S.
\end{matrix}\right.$$
Then spectra and eigenspaces of the normal Cayley graph $\Gamma(G, S)= \Gamma(G; \alpha)$ can be explicitly drawn from the above formulas.

\section{Construction of non-normal Cayley graphs}

Let $K= \{e= k_1, k_2, \ldots, k_m\}$ be a subgroup of $G$ and $H= \{e= h_1, h_2, \ldots, h_l\}$ a set of all left coset representatives of $K$ in $G$, called a left \textit{transversal}. Arranging the vertices of $\Gamma(G; \alpha)$ as
$$e, k_2, \ldots, k_m; h_2, h_2k_2, \ldots, h_2k_m; \dots; h_l, h_lk_2, \dots, h_lk_m,$$
the adjacency matrix with respect to the above ordering is a block matrix
$$\{\text{adj}(\Gamma(K; \beta_{ij}))\}_{i,j= 1}^l, \text{ where } \beta_{ij}(k)= \alpha(h_jkh_i^{-1}), \forall k\in K.$$

If $H$ is a subgroup of $G$ and $G$ is a semidirect product of $K$ and $H$ denoted by $K\rtimes H$, then $G$ is called a \textit{split extension} of $H$ by $K$. Let $\rho_1^H, \ldots, \rho_r^H$ be irreducible unitary representations of degrees $d_1^H, \ldots, d_r^H$ of $H$ with characters $\chi_1^H, \ldots, \chi_r^H$, and respectively $\rho_1^K, \ldots, \rho_s^K$ irreducible unitary representations of degrees $d_1^K, \ldots, d_s^K$ of $K$ with characters $\chi_1^K, \ldots, \chi_s^K$. Let $C_i, 1\leq i\leq r$ be the conjugacy classes of $H$ and $h_{C_i}\in C_i$. We can obtain the following main result.

\begin{theorem}\label{thm31}
Suppose that $G= K\rtimes H$, and the function $\alpha$ on $G$ satisfies
$$\alpha(hgkg^{-1})= \alpha(hk), \alpha(h'hh'^{-1}k)= \alpha(hk), \forall g\in G, \forall k\in K, \forall h,h'\in H,$$
then the spectrum of the Cayley color graph $\Gamma(G; \alpha)$ is given by
$$\frac{1}{d_u^Hd_v^K}\sum_{i= 1}^r(|C_i|\chi_u^H(h_{C_i})\sum_{k\in K}\alpha(h_{C_i}k)\chi_v^K(k)), 1\leq u\leq r, 1\leq v\leq s,$$
and the corresponding eigenspaces are
$$(\rho_u^H)_{ij}\otimes (\rho_v^K)_{i'j'}, 1\leq i, j\leq d_u^H, 1\leq i', j'\leq d_v^K, 1\leq u\leq r, 1\leq v\leq s.$$
\end{theorem}

\begin{proof}
If $h_t= h_jh_i^{-1}, 1 \leq i, j, t\leq l$, then for any $k, k'\in K$ we see that
$$\beta_{ij}(k)= \alpha(h_jkh_i^{-1})= \alpha((h_jh_i^{-1})h_ikh_i^{-1})= \alpha(h_tk)= \beta_{1t}(k)$$
and
$$\beta_{1t}(k'kk'^{-1})= \alpha(h_tk'kk'^{-1})= \alpha(h_tk)= \beta_{1t}(k),$$
which imply that $\beta_{ij}, 1\leq i, j\leq l$ are all class functions on $K$ and
$${\rm adj}(\Gamma(G; \alpha))= \sum_{t= 1}^l (\rho_{\rm reg}^H(h_t))^T\otimes (\widehat{\beta_{1t}}(\rho_{\rm reg}^K))^T,$$
where $\rho_{\rm reg}^H$ and $\rho_{\rm reg}^K$ are left regular representations of $H$ and $K$ respectively.

Let $\beta_{1C_i}(k)= \alpha(h_{C_i}k)$, $\forall k\in K$. If $h_t= hh_{C_i}h^{-1}, h\in H$, then
$$\beta_{1t}(k)= \alpha(h_tk)= \alpha(hh_{C_i}h^{-1}k)= \alpha(h_{C_i}k)= \beta_{1C_i}(k).$$
Hence we have
$${\rm adj}(\Gamma(G; \alpha))= \sum_{i= 1}^r (\sum_{h\in C_i}(\rho_{\rm reg}^H(h))^T\otimes (\widehat{\beta_{1C_i}}(\rho_{\rm reg}^K))^T).$$

Suppose that
$$P_H= ((\rho_1^H)_{11}, \ldots, (\rho_1^H)_{d_1^H 1}, (\rho_1^H)_{12}, \ldots, (\rho_1^H)_{d_1^Hd_1^H}, (\rho_2^H)_{11}, \ldots, (\rho_2^H)_{d_2^Hd_2^H}, \ldots, (\rho_r^H)_{d_r^Hd_r^H}),$$
$$P_K= ((\rho_1^K)_{11}, \ldots, (\rho_1^K)_{d_1^K 1}, (\rho_1^K)_{12}, \ldots, (\rho_1^K)_{d_1^Kd_1^K}, (\rho_2^K)_{11}, \ldots, (\rho_2^K)_{d_2^Kd_2^K}, \ldots, (\rho_s^K)_{d_s^Hd_s^K}),$$
then by Theorem \ref{thm21} we have
\begin{equation*}
\aligned
&(P_H\otimes P_K)^{-1}{\rm adj}(\Gamma(G; \alpha))(P_H\otimes P_K)\\
=& \sum_{i= 1}^r (P_H^{-1}(\sum_{h\in C_i}\rho_{\rm reg}^H(h))^TP_H\otimes P_K^{-1}(\widehat{\beta_{1C_i}}(\rho_{\rm reg}^K))^TP_K)\\
=& \sum_{i= 1}^r {\rm diag}(\lambda_{1i}^H I_{(d_1^H)^2}, \ldots, \lambda_{ri}^H I_{(d_r^H)^2})\otimes {\rm diag}(\sigma_{1i}^K I_{(d_1^K)^2}, \ldots, \sigma_{si}^KI_{(d_s^K)^2}),
\endaligned
\end{equation*}
where
$$\lambda_{ui}^H= \frac{1}{d_u^H}|C_i|\chi_u^H(h_{C_i}), u= 1, \ldots, r, \sigma_{vi}^K= \frac{1}{d_v^K}\sum_{k\in K}\alpha(h_{C_i}k)\chi_v^K(k), 1\leq v\leq s.$$
Therefore the spectrum of the Cayley color graph $\Gamma(G; \alpha)$ is given by
$$\sum_{i= 1}^r \lambda_{ui}^H\sigma_{vi}^K, 1\leq u\leq r, 1\leq v\leq s,$$
and the corresponding eigenspaces are
$$(\rho_u^H)_{ij}\otimes (\rho_v^K)_{i'j'}, 1\leq i, j\leq d_u^H, 1\leq i', j'\leq d_v^K, 1\leq u\leq r, 1\leq v\leq s.$$
\end{proof}

\begin{remark}\label{rem32}
We consider the symmetric group $S_4= A_4\rtimes \langle(12)\rangle$, where $A_4$ is the alternating group and $\langle(12)\rangle$ denotes the subgroup generated by the transposition $(12)$. The class function $\alpha$ on $S_4$ takes defferent values in defferent conjugacy classes. Then
$$\alpha((12)(1432)(234)(1432)^{-1})= \alpha((23))\neq \alpha((1234))= \alpha((12)(234)).$$
Hence $\alpha$ does not satisfy the conditions of Theorem \ref{thm31}. If $K$ and $H$ are both abelian subgroups of $G$, then class functions on $G$ obviously satisfy the conditions of Theorem \ref{thm31}. The following conclusion implies that there exist many non-class functions that satisfy the conditions of Theorem \ref{thm31}. 
\end{remark}

Let $C_m$ and $C_l$ be cyclic groups of order $m$ and $l$ generated by $k$ and $h$ respectively. If $G$ is a split extension of $C_l$ by $C_m$, i.e. $G= C_m\rtimes C_l$, we call $G$ a \textit{split metacyclic group}. The following conclusion can be obtained directly from Theorem \ref{thm31}.
\begin{corollary}\label{col33}
Suppose that $G$ is a split metacyclic group, i.e. $G= C_m\rtimes C_l$, and
$$S= S_0\cup hS_1\cup \cdots \cup h^{l-1}S_{l-1},$$
where $S_i\subseteq C_m$ and $hS_i h^{-1}\subseteq S_i, 0\leq i\leq l-1$. Then the spectrum of the Cayley graph $\Gamma(G, S)$ is given by
$$\sum_{t= 0}^{l-1}(e^{2\pi iut/l}\sum_{k^s\in S_t}e^{2\pi ivs/m}), 0\leq u\leq l-1, 0\leq v\leq m-1,$$
and the corresponding eigenspaces are
$$\frac{1}{\sqrt{l}}(1, e^{2\pi iu/l}, \ldots, e^{2\pi iu(l-1)/l})^T\otimes \frac{1}{\sqrt{m}}(1, e^{2\pi iv/m}, \ldots, e^{2\pi iv(m-1)/m})^T.$$
\end{corollary}

\begin{remark}\label{rem34}
For $G= C_m\rtimes C_l$ with $hkh^{-1}= k^r, 1< r< m$, suppose that $S= (C_m\setminus\{e\})\cup \{h\}\cup \{h^{-1}\}$. Then $S$ satisfy the conditons of Corollary \ref{col33}, but is not closed under conjugation since $k^rhk^{-r}= hk^{1-r}\notin S$. Therefore We can easily choose $S_i$ closed under the conjugation of $h$ to construct many non-normal Cayley graphs, and exactly give spectra and eigenspaces of  these graphs by Corollary \ref{col33}.
\end{remark}

\begin{remark}\label{rem35}
Let $H$ and $K$ be subgroups of a group $G$ with identity element $e$. If $G= HK$ and $H\cap K= \{e\}$, then $G$ is said to be a \textit{Zappa-Szép product} of $H$ and $K$. As a generalization of the direct and semidirect products, we can consider using it construct non-normal Cayley graphs. In addition, a composition series of a group can also be used to construct non-normal Cayley graphs.
\end{remark}

\vspace{2mm}
\noindent
{\bf Acknowledgments.} This research is partially supported by the National Natural Science Foundation of China (No. 12271210) and the Natural Science Foundation
of Fujian Province, China (No. 2021J01862).

\end{document}